\documentclass{amsart}
\usepackage{amsmath,amssymb,amsthm, amscd}
\usepackage{pifont,mathrsfs}
\usepackage{array,lastpage}
\usepackage{enumerate,xspace,pifont,shadow}
\usepackage{mathrsfs}
\usepackage[T1]{fontenc}
\usepackage{mathtools}

\DeclareSymbolFont{AMSb}{U}{msb}{m}{n}
\DeclareMathSymbol{\Z}{\mathbin}{AMSb}{"5A}
\DeclareMathSymbol{\R}{\mathbin}{AMSb}{"52}
\DeclareMathSymbol{\N}{\mathbin}{AMSb}{"4E}
\DeclareMathSymbol{\Q}{\mathbin}{AMSb}{"51}

\def\Ind{\setbox0=\hbox{$x$}\kern\wd0\hbox to 0pt{\hss$\mid$\hss}
\lower.9\ht0\hbox to 0pt{\hss$\smile$\hss}\kern\wd0}

\def\Notind{\setbox0=\hbox{$x$}\kern\wd0\hbox to 0pt{\mathchardef
\nn=12854\hss$\nn$\kern1.4\wd0\hss}\hbox to
0pt{\hss$\mid$\hss}\lower.9\ht0 \hbox to
0pt{\hss$\smile$\hss}\kern\wd0}

\newtheorem{thm}{Theorem}[section]
\newtheorem{lem}[thm]{Lemma}

\theoremstyle{definition}
\newtheorem{definition}[thm]{Definition}

\theoremstyle{remark}

\theoremstyle{remark}
\newtheorem{example}[thm]{Example}

\theoremstyle{remark}
\newtheorem{claim}[thm]{Claim}

\theoremstyle{remark}
\newtheorem{conjecture}[thm]{Conjecture}

\theoremstyle{remark}

\theoremstyle{remark}

\begin{document}
\bibliographystyle{amsplain}

\title[Parametric expansions of Presburger arithmetic]{Bounding quantification in parametric expansions of Presburger arithmetic}
\author{John Goodrick}

\maketitle

\begin{abstract}
Generalizing Cooper's method of quantifier elimination for Presburger arithmetic, we give a new proof that all parametric Presburger families $\{S_t : t \in \N\}$ (as defined by Woods in \cite{woods1}) are definable by formulas with polynomially bounded quantifiers in an expanded language with predicates for divisibility by $f(t)$ for every polynomial $f \in \Z[t]$. In fact, this quantifier bounding method works more generally in expansions of Presburger arithmetic by multiplication by scalars $\{\alpha(t):  \alpha \in R, t \in X\}$ where $R$ is any ring of functions from $X$ into $\Z$.

\end{abstract}

\section{Introduction}

In this note we study expansions of Presburger arithmetic which are motivated by the combinatorial study of the following families, defined by Woods in \cite{woods1}:

\begin{definition}
\label{param_pres}
Let $t$ be a variable ranging over $\N$ and fix $d \in \N$. A family $\{S_t : t \in \N\}$ of sets $S_t \subseteq \Z^d$ is a \emph{parametric Presburger family} if it is defined by an inequality of the form $$\sum_{i=1}^d a_i(t) x_i \leq b(t)$$ where $a_1, \ldots, a_d, b \in \Z[t]$, or if it can be constructed from such inequalities by a finite sequence of Boolean operations and coordinate projections.
\end{definition}

When defining a parametric Presburger family, quantifying over the ``parameter'' variable $t$ is not allowed; we may only quantify over the $\Z$-valued variables $x_1, \ldots, x_d$. 

Woods conjectured that the counting function $g(t) = |S_t|$ of a parametric Presburger family is eventually quasi-polynomial, that is, if there exist $m$ and polynomials $g_1, \ldots, g_m \in \Q[t]$ such that $g(t) = g_i(t)$ if $t \equiv_m i$ and $t \gg 0$. This conjecture was a significant generalization of Ehrhart's theorem (see \cite{ehrhart}) asserting that if $S_t$ is the $t$-th dilate of a fixed polytope with vertices in $\Z^d$ then $g(t)$ is a polynomial. Woods further conjectured that $\max(S_t)$ (defined whenever $S_t$ is finite) is an eventual quasi-polynomial function of $t$, which would give a positive answer to a recent question of Roune and Woods on the \emph{parametric Frobenius problem} from combinatorics (see \cite{roune_woods}).

These combinatorial conjectures of Woods were eventually proved by Bogart, Woods, and the present author \cite{BGW} using the results of the current paper in two key steps. Given that Woods had previously verified his conjectures in the case when the family $S_t$ is definable by a \emph{quantifier-free} parametric Presburger formula, the idea that one should try to eliminate quantifiers in some reasonable language was clear from a logical perspective. As it turns out, we still do not know of a good language in which outright quantifier elimination can be achieved (see Example~\ref{intervals} below and the discussion which follows), but a logical reduction to parametric Presburger formulas with bounded quantifiers did turn out to be useful in establishing the conjectures of Woods, hence the inspiration for this paper.  For a detailed outline of the proof of these conjectures, the interested reader is referred to the introduction of \cite{BGW}.

Now we will make precise the languages we will consider in this article, beginning with the classical language $\mathcal{L}^+_{Pres}$ for quantifier elimination in Presburger arithmetic:

\begin{definition}

\begin{enumerate}

\item $\mathcal{L}_{Pres} = \{0, 1, <, -, +\}$ is the usual first-order language with constant symbols for $0$ and $1$, a binary relation for the ordering, a unary operation symbol $-$ for negation, and a binary operation symbol $+$ for addition.

\item $\mathcal{L}^+_{Pres} = \mathcal{L}_{Pres} \cup \{D_n : n \in \N\}$ where $D_n$ is a unary predicate representing divisibility by $n$. (By convention, $D_0(x)$ is always false.)

\item $\mathcal{L}_{\Z[t]} = \mathcal{L}_{Pres} \cup \{f_\alpha : \alpha \in \Z[t]\}$, where $f_\alpha$ is a unary function symbol representing scalar multiplication by $\alpha(t)$ and $t \in \N$.

\item $\mathcal{L}^+_{\Z[t]} = \mathcal{L}_{\Z[t]} \cup \{D_\alpha : \alpha \in \Z[t] \}$ where each $D_\alpha$ is a unary predicate symbol representing divisibility by $\alpha(t)$ for some parameter $t \in \N$.

\end{enumerate}

\end{definition}

Unfortunately, quantifier elimination in the full first-order $\mathcal{L}_{\Z[t]}^+$-theory of the integers is not possible; see Example~\ref{intervals} below.

What we will show is that any first-order $\mathcal{L}_{\Z[t]}^+$-formula is equivalent to one with polynomially bounded quantifiers. We will prove this in the following more general context.

\begin{definition}
Let $X$ be any nonempty set and let $R$ be any ring of functions $f: X \rightarrow \Z$ which includes all constant functions.

\begin{enumerate}

\item $\mathcal{L}_R = \mathcal{L}_{Pres} \cup \{f_\alpha : \alpha \in R\}$, where $f_\alpha$ is a unary function symbol representing scalar multiplication by $\alpha(t)$ and $t \in X$.

\item $\mathcal{L}^+_R = \mathcal{L}_R \cup \{D_\alpha : \alpha \in R \}$ where each $D_\alpha$ is a unary predicate symbol representing divisibility by $\alpha(t)$ for some parameter $t \in X$.

\end{enumerate}

\end{definition}

For ease of notation, instead of ``$f_\alpha(s)$'' we will usually write ``$\alpha \cdot s$'' or ``$\alpha(t) \cdot s$'' (if we want to emphasize the role of the parameter $t$). We will also use ``$\alpha$'' to denote the term $f_\alpha(1)$ when $\alpha \in R$, which should cause no confusion.

Now we can state our main result:

\begin{thm}
\label{qe_main}

Every first-order $\mathcal{L}^+_R$-formula is logically equivalent to an $\mathcal{L}^+_R$-formula with $R$-bounded quantifiers.
\end{thm}

In the next section we will define ``$R$-bounded quantifiers'' and make precise the notion of logical equivalence that we are using. 

After writing our proof of Theorem~\ref{qe_main}, we found out that Lasaruk and Sturm had already proven a roughly equivalent result in \cite{lasaruksturmweakQE}. Our proof given here is significantly different; for instance, at no point do we convert our formula into disjunctive normal form. See below in the next section for a detailed comparison.

We also give a criterion for when an $\mathcal{L}^+_R$-formula is equivalent to a quantifier-free formula in $\mathcal{L}_R^+$. This will be useful in forthcoming work (joint with Bogart and Woods) in which we will prove Woods's conjecture that $|S_t|$ is eventually quasi-polynomial when $R = \Z[t]$.

\begin{thm}
\label{qe_criterion}
Suppose that $\varphi(\overline{x})$ is an $\mathcal{L}^+_R$-formula satisfying:

\begin{enumerate}
\item If the divisibility predicate $D_\alpha$ occurs in $\varphi(\overline{x})$, then $\alpha$ is a constant function; and
\item $\varphi$ does not contain any term of the form $f_\alpha(s)$ where $s$ contains a variable that is within the scope of a quantifier and $\alpha$ is not a constant.
\end{enumerate}

Then $\varphi(\overline{x})$ is logically equivalent to a quantifier-free $\mathcal{L}^+_R$-formula which also satisfies condition (1).

\end{thm}

For example, consider the case when $X = \N$ and $R = \Z[t]$. Then condition (2) does not apply to $\exists y \left[ t \cdot y = x\right]$ (which is equivalent to the divisibility condition $D_t(x)$), but it does apply to $\exists y \left[ t \cdot x = y \right]$ (which trivially holds for any value of $x$). Condition (2) allows $\varphi$ to contain subformulas such as $x \leq p(t)$ where $x$ is any variable and $p(t) \in \Z[t]$.

To prove Theorems~\ref{qe_main} and \ref{qe_criterion}, we adapt Cooper's quantifier elimination procedure from \cite{cooper} for classical Presburger arithmetic. Note that if we fix a value for the parameter $t$, then a bounded quantifier such as $\exists x \left[(0 \leq x \leq f(t)) \wedge \ldots \right]$ is equivalent to a finite disjunction, and in this case our procedure reverts to Cooper's algorithm. What needs to be checked carefully is the dependence on $t$.

In Section 2, we will define everything carefully and compare our results with previous work. Section 3 is devoted to the proofs of Theorems~\ref{qe_main} and \ref{qe_criterion}.

We would like to thank Tristram Bogart for many useful discussions as well as first bringing Woods's work to our attention, and Kevin Woods for showing us the prior work by Lasaruk and Sturm. We also thank the anonymous referee for their careful reading and many useful suggestions for improving an earlier version of this paper.

\section{$R$-parametric Presburger families and $R$-bounded quantifiers}

In this section we set notation and carefully define the notions of truth and logical equivalence for first-order $\mathcal{L}_R$-formulas.

Throughout, we use standard terminology from first-order logic (atomic formulas, variables, \emph{et cetera}). A ``formula'' always means a first-order formula. We work in the languages $\mathcal{L}_R$ and $\mathcal{L}_R^+$ defined above. 

$R$ always denotes some subring of the ring of all functions from $X$ into $\Z$ where $X$ is some fixed domain, and we assume that $R$ contains every constant function from $X$ into $\Z$.

The truth definition for first-order formulas in $\mathcal{L}_R$ or in $\mathcal{L}^+_R$ is not quite the standard one, since it will depend on a \emph{parameter} $t \in X$ at which all functions $\alpha \in R$ are to be evaluated.

\begin{definition}
An \emph{$R$-parametrized Presburger formula} is a first-order formula in the language $\mathcal{L}^+_R$. The letter $t$ always denotes a \emph{parameter} in $X$ used for evaluating $R$-parametrized Presburger formulas.

Given $t \in X$ and an $\mathcal{L}^+_R$-formula $\varphi$, we define the $\mathcal{L}^+_{Pres}$-formula $\varphi_t$ to be the translation of $\varphi$ defined recursively so that each term $\alpha \cdot s$ in $\varphi$, where $\alpha \in R$ and $s$ is a term, is replaced by one of the following:

\begin{enumerate}
\item If $\alpha(t) > 0$, then $\alpha \cdot s $ is replaced in $\varphi_t$ by $s + s + \ldots + s$ with $\alpha(t)$ repetitions of $s$;
\item if $\alpha(t) < 0$, then $\alpha \cdot s$ is replaced in $\varphi_t$ by $-(s + s + \ldots + s)$ with $\alpha(t)$ repetitions of $s$;
\item if $\alpha(t) = 0$, then $\alpha \cdot s$ is replaced in $\varphi_t$ by the constant symbol $0$.
\end{enumerate}

Similarly, the atomic formula $D_\alpha(s)$ in $\varphi$ is replaced in $\varphi_t$ by $D_{\alpha(t)}(s)$, using the convention that if $\alpha(t) = 0$ then $D_{\alpha(t)}(s)$ is always false.

\end{definition}

The definition of $\varphi_t$ above allows us to talk about the \emph{truth} of $\mathcal{L}^+_R$-formulas relative to a parameter $t \in X$: given an $\mathcal{L}^+_R$-formula $\varphi(x_1, \ldots, x_d)$ whose free variables are contained in $\{x_1, \ldots, x_d\}$, $t \in X$, and $(k_1, \ldots, k_d) \in \Z^d$, we will write $$\models \varphi_t(k_1, \ldots, k_d)$$ just in case the $\mathcal{L}^+_{Pres}$-formula $\varphi_t$ is true in $\Z$ with the variable $x_i$ evaluated as $k_i$.

\begin{definition}
\label{equiv}
Given two $\mathcal{L}^+_R$-formulas $\varphi(x_1, \ldots, x_d)$ and $\psi(x_1, \ldots, x_d)$, we write $$\varphi(x_1, \ldots, x_d) \models \psi(x_1, \ldots, x_d)$$ just in case for \emph{every} $t \in X$, we have $$\varphi_t(x_1, \ldots, x_d) \models \psi_t(x_1, \ldots, x_d),$$ or in other words, for every $t \in X$ and every $(k_1, \ldots, k_d) \in \Z^d$, $$\models \varphi_t(k_1, \ldots, k_d) \Leftarrow \models \psi_t(k_1, \ldots, k_d).$$ We say that $\varphi(x_1, \ldots, x_d)$ and $\psi(x_1, \ldots, x_d)$ are \emph{logically equivalent} if $\varphi(x_1, \ldots, x_d) \models \psi(x_1, \ldots, x_d)$ and $\psi(x_1, \ldots, x_d) \models \varphi(x_1, \ldots, x_d)$.

\end{definition}

It is clear that every $\mathcal{L}^+_R$-formula is logically equivalent to some $\mathcal{L}_R$-formula.

\begin{definition}
An \emph{$R$-parametrized Presburger family} is an $X$-indexed family of sets $S_t \subseteq \Z^d$ for some fixed $d \in \N$ such that there is an $\mathcal{L}^+_R$-formula $\varphi(x_1, \ldots, x_d)$ such that for each $t \in X$, $$S_t = \{ (k_1, \ldots, k_d) \in \Z^d : \, \, \, \models \varphi_t(k_1, \ldots, k_d) \}.$$
\end{definition}

\begin{definition}
Given an $\mathcal{L}^+_R$-formula $\varphi(x_1, \ldots, x_d, z)$ whose free variables are among $\{x_1, \ldots, x_d, z\}$ and $\alpha \in R$, we denote by $$\bigvee_{z=0}^{\alpha(t)} \varphi(x_1, \ldots, x_d, z)$$ the $\mathcal{L}^+_R$-formula $$\exists z \left[0 \leq z \leq \alpha(t) \wedge \varphi(x_1, \ldots, x_d, z) \right],$$ in which $z$ is part of an \emph{$R$-bounded existential quantifier}. 

An \emph{$R$-bounded universal quantifier} applied to $\varphi$ yields $$\forall z \left[ 0 \leq z \leq \alpha(t) \rightarrow \varphi(x_1, \ldots, x_d, z) \right]$$ for some $\alpha \in R$.

An \emph{$\mathcal{L}^+_R$-formula with $R$-bounded quantifiers} is a member of the smallest class of $\mathcal{L}^+_R$-formulas containing all atomic formulas and closed under Boolean combinations and the formation of $R$-bounded quantifiers (both existential and universal).

\end{definition}

\begin{example}
\label{intervals}
Even when $R = \Z[t]$ and $X = \N$, it is not the case that every $\mathcal{L}_R$-formula is logically equivalent to a quantifier-free formula in $\mathcal{L}^+_R$. For example, consider an $\mathcal{L}_R$-formula $\varphi(x)$ which defines the set $$S_t := \bigcup_{i=0}^{t-1} \left[ 2i \cdot t, \,  (2i + 1) \cdot t\right].$$ Suppose that $\theta(x)$ is a quantifier-free $\mathcal{L}^+_R$-formula such that for every $t \in \N$, $\models \forall x \left[\theta_t(x) \rightarrow \varphi_t(x) \right].$ We may assume that $$\theta(x) = \bigvee_{1 \leq i \leq n} \psi_i(x)$$ where each formula $\psi_i$ is a conjunction of atomic formulas and negations of atomic formulas. If $i \in \{1, \ldots, n\}$ and there is some $t > 1$ such that $(\psi_i)_t(\Z)$ intersects more than one of the intervals $I_j := \left[2j \cdot t, \, (2j + 1) \cdot t\right]$, then one of the conjuncts of $\psi_i(x)$ must be a divisibility condition $D_{\alpha(t)}(x)$, and furthermore $\alpha(t)$ must have degree $1$ (otherwise $\psi_i(x)$ could not imply $\varphi(x)$); but then $D_{\alpha(t)}(\Z)$ can intersect each $I_j$ \emph{at most twice}. (Note that if $\alpha(t) = t-c$ where $c$ is a constant, then  $D_{\alpha(t)}(\Z)$ may intersect an interval $I_j$ more than once.) So every $(\psi_i)_t(\Z)$ is either always a subset of a single $I_j$, or else always intersects each $I_j$ at most once, and when $t > n$ the set $S_t$ (which is comprised of $t$ disjoint intervals of length $t$) cannot be covered by all of $\theta_t(\Z)$.
\end{example}

Recently Petr Glivick\'y has made a detailed classification of expansions of Presburger arithmetic by a unary function representing multiplication by a nonstandard element (see \cite{glivicky}). His methods may be adaptable to parametric Presburger families, and they suggest the following:

\begin{conjecture}
If $R = \Z[t]$ and $X = \N$, then every $\mathcal{L}_R$-formula is logically equivalent to a quantifier-free formula in the language $$\mathcal{L}'_R = \mathcal{L}_R \cup \{g_{\alpha(t)} : \alpha \in R\}$$ where $g_{\alpha(t)}$ is a unary function symbol interpreted as $$g_{\alpha(t)}(x) = \max \{q \in \Z : q \cdot |\alpha(t)| \leq x \}$$ (letting $g_{\alpha(t)}(x) = 0$ in case $\alpha(t) = 0$).
\end{conjecture}

In other words, $g_{\alpha(t)}(x)$ gives the ``floor function'' applied to the quotient $$\frac{x}{|\alpha(t)|}$$ (except in the case when $\alpha(t) = 0$). The equivalence between $\mathcal{L}'_R$-formulas is defined in the natural way just as in Definition~\ref{equiv} above. Note that the divisibility predicate $D_\alpha$ can be defined from this function without quantifiers, since $D_{\alpha(t)}(x)$ is true if and only if $\alpha(t) \neq 0$ and $x = \alpha(t) \cdot g_{\alpha(t)} (x).$

\subsection{Comparison with previous work}

After we had already found a proof of Theorem~\ref{qe_main}, we became aware of similar prior results by Weispfenning \cite{weis} and Lasaruk and Sturm \cite{lasaruksturmweakQE}. Here we briefly summarize their work and what is new in our work.

In \cite{weis}, Weispfenning introduced \emph{Uniform Presburger Arithmetic} (UPA), a two-sorted extension of Presburger arithmetic which is essentially the same as $R$-parametrized Presburger arithmetic. In Weispfenning's language, there is a \emph{scalar sort} (corresponding to our ring $R$ of functions) and a \emph{vector sort} which is a model of Presburger arithmetic, as well as a binary scalar multiplication operation $(\alpha, x) \mapsto \alpha \cdot x$ defined for scalars $\alpha$ and elements $x$ of the vector sort. Quantification is permitted over the vector sort but not over the scalar sort, just as we do not allow quantification over $R$. Assuming that the scalar sort is closed under the ring operations plus maximum and least-common-multiple operators, Weispfenning shows that any formula in UPA is equivalent to one with bounded quantifiers and congruence relations. Theorem~\ref{qe_main} is more general: we assume only that $R$ is closed under the ring operations. 

We should remark that although Weispfenning's  bounded quantifier elimination result for UPA from \cite{weis} is very similar to our main result, the procedure he uses is different. Our method is an adaptation of Cooper's algorithm, in which one of the key steps is to eliminate non-unitary coefficients in front of quantified variables before eliminating them. For example, suppose our original input is the formula $$\exists x \left[ 2x \leq a_1 \wedge 3x \equiv_5 a_2 \right]$$ (where $a_1, a_2$ are terms not involving $x$ and ``$s \equiv_m t$'' can be taken as shorthand for $D_m(t - s)$). Then Cooper's algorithm would first use the substitution $y = 6x$ to replace this by the equivalent formula $$\exists y \left[  y \leq 3 a_1 \wedge y \equiv_{10} 2 a_2 \right]$$ in preparation for eventually eliminating the quantifier $\exists y$. (Of course the formula happens to be trivially always true in this case, but we chose a simple example to illustrate the syntactic procedure.)  In contrast, the algorithm in \cite{weis} applied to the same formula would use a different linear substitution, $2x = a_1 + z$, to eliminate the unbounded quantifier $\exists x$ in one step, yielding (modulo some trivial reductions) the equivalent formula $$\left(0 \leq a_1 \wedge 0 \equiv_5 a_2 \right) \vee \exists z \in [-20, 0] \, \left( a_1 + z \equiv_2 0 \wedge 3 a_1 + 3z \equiv_{10} 2 a_2 \right)$$ with the new variable $z$ appearing in a bounded quantifier.

In \cite{lasaruksturmweakQE}, Lasaruk and Sturm studied $\Z$ with the ordered ring language with full binary multiplication, not just unary multiplication by a coefficient $\alpha(t)$.  They also have ternary relation symbols for $x \equiv_m y$ where any terms can be substituted for $x, y,$ and $m$. However, they restrict consideration to \emph{linear formulas}.  An atomic formula $\theta$ is called \emph{linear} in the set of variables $\{x_1, \ldots, x_n\}$ if it contains no products $x_i \cdot x_j$ (including $x_i \cdot x_i$), and if $\theta$ is a congruence condition $s \equiv_m s'$, then no variable $x_i$ occurs in the term $m$; a general formula is linear if each of its atomic subformulas is linear in the set of its variables which are within the scope of a quantifier. The \emph{full linear theory of $\Z$} is the theory restricted to linear formulas.
 
The main result of \cite{lasaruksturmweakQE} is what they call ``weak quantifier elimination'' for the full linear theory of $\Z$: any linear formula is logically equivalent to a linear formula with only bounded quantifiers. For them, ``bounded quantifiers'' have a slightly more general meaning than in this paper: they allow generalized disjunctions of the form $$\bigvee_{k : \models \psi(\overline{x}, k)} \varphi(\overline{x}, k)$$ over any given formula $\psi(\overline{x}, z)$, as long as $\psi(\overline{a}, \Z)$ is finite for any tuple $\overline{a}$ from $\Z$ of the right length.

In the linear formulas $\varphi(x_1, \ldots, x_n)$ considered by Lasaruk and Sturm, it is possible that there are terms involving products $x_i \cdot x_j$ of the free variables. There is no obstacle to applying our proof of Theorem~\ref{qe_main} to handle such formulas; our proof of the crucial Lemma~\ref{prepared}, for example, would only fail if there were products $y \cdot y$ between a quantified variable $y$ and itself, which is disallowed in linear formulas.

Theorem~\ref{qe_main} could be derived from Lasaruk and Sturm's result as follows: if $\varphi(\overline{x})$ is an $\mathcal{L}_R^+$-formula which involves the elements $\alpha_1, \ldots, \alpha_m \in R$ in the scalar multiplications $f_{\alpha_i}$ or $D_{\alpha_i}$, then we can replace each $\alpha_i$ by a new free variable $y_i$ and obtain a formula $\varphi'(\overline{x}, \overline{y})$ which is linear in the sense of Lasaruk and Sturm. For instance, the formula $D_{\alpha_i}(s)$ can be translated as $\exists z \left[z \cdot y_i = s\right]$ where $y_i$ replaces $\alpha_i$ and $z$ is some new variable. Then the Lasaruk-Sturm result implies that $\varphi'(\overline{x}, \overline{y})$ is equivalent to a linear formula with only bounded quantifiers, and this could be translated back into an $\mathcal{L}^+_R$-formula with bounded quantifiers.

Finally, we note that for the proof of Woods' combinatorial conjectures in \cite{BGW}, we needed not only the bounded quantifier elimination result Theorem~\ref{qe_main} but also the criterion in Theorem~\ref{qe_criterion} which gives a sufficient condition for eliminating quantifiers outright in $\mathcal{L}^+_R$. This latter result can be read off easily from our proof of Theorem~\ref{qe_main} presented below, but it is not (in any immediately apparent way) a corollary of the other quantifier-bounding procedures we found in the literature.

\section{Proofs of Theorems~\ref{qe_main} and \ref{qe_criterion}}

In this section, we will prove Theorem~\ref{qe_main} via a sequence of lemmas (up to and including Lemma~\ref{exists_elimination}), and then we explain how the same argument can be used to prove Theorem~\ref{qe_criterion}. Our strategy is to generalize the method of elimination for ordinary Presburger arithmetic discovered by Cooper in \cite{cooper}. A feature of Cooper's quantifier elimination algorithm which is useful for us is that it does not involve reducing a quantifier-free formula to disjunctive or conjunctive normal form.

Throughout this section, we will fix some ring $R$ of functions from $X$ into $\Z$ and some $\mathcal{L}^+_R$-formula with $R$-bounded quantifiers $\varphi(x_1, \ldots, x_n, y)$, with the aim of finding an $\mathcal{L}^+_R$-formula with $R$-bounded quantifiers which is logically equivalent to $\exists y \left[ \varphi(x_1, \ldots, x_n, y) \right]$. We will further assume that the variable $y$ does not occur in any $R$-bounded quantifier in $\varphi$ (after renaming quantified variables as needed).

Instead of putting $\varphi$ into disjunctive normal form, we will make it satisfy the following condition:

\begin{definition}
\label{prepared}
An $\mathcal{L}^+_R$-formula with $R$-bounded quantifiers is called \emph{normalized (in $y$)} if it can be constructed using only the positive Boolean operators ($\vee$ and $\wedge$) and $R$-bounded quantifiers from the following types of \emph{basic formulas}:

\begin{enumerate}
\item $y < a$,
\item $b < y$,
\item $D_\alpha(y + c)$,
\item $\neg D_\beta(y + d)$, and
\item atomic or ``negatomic'' (negations of atomic) formulas not involving the variable $y$,
\end{enumerate}

where $a, b, c$ and $d$ are terms in $\mathcal{L}^+_R$ which do not contain the variable $y$ and $\alpha, \beta \in R$.

\end{definition}

\begin{lem}
\label{prepared_lem}
There is some formula $\widetilde{\varphi}(x_1, \ldots, x_n, y')$  which is normalized in $y'$ and has $R$-bounded quantifiers such that  $\exists y \, \varphi(x_1, \ldots, x_n, y) $ is logically equivalent to $\exists y' \, \widetilde{\varphi}(x_1, \ldots, x_n, y')  $.

\end{lem}

\begin{proof}
Essentially we would like to repeat an idea from \cite{cooper} and rewrite each atomic subformula in terms of $y' = \nu \cdot y$, where $\nu \in R$ is a common multiple of all the coefficients of $y$ occurring in $\varphi$. However, we must do something a little more complicated than this to take into account that some of these coefficients may be $0$ for certain values of $t \in X$.

First, we move all negations in $\varphi$ inward as far as possible, applying De Morgan's law repeatedly, including to conjunctions and disjunctions of $\alpha(t)$ formulas where $\alpha \in R$, then eliminate double negations. Thus we may assume that $\varphi$ is constructed from atomic and negatomic formulas via \emph{positive} Boolean combinations and $R$-bounded quantifiers.

We may further assume that each atomic or negatomic subformula of $\varphi$ is an instance of one of the following types, where Greek letters $\alpha, \ldots \theta \in R$ and Latin letters $a, \ldots, f$ are terms not involving the variable $y$:

\begin{enumerate}
\item $\alpha \cdot  y < a$,
\item $b < \beta \cdot y$,
\item $D_\gamma( \delta \cdot y + c)$,
\item $\neg D_\epsilon( \zeta \cdot y + d)$, 
\item $\eta \cdot y = e$,
\item $\theta \cdot y \neq f$, and
\item atomic or negatomic formulas not involving the variable $y$.
\end{enumerate}

Next, eliminate atomic formulas of type 5 from $\varphi$ by replacing them with $$\left( \eta \cdot y > e-1\right) \wedge \left(\eta \cdot y < e+1\right) $$ (a conjunction of formulas of type 1 and 2), and similarly replace all negatomic formulas of type 6 by disjunctions of inequalities of types 1 and 2. 

Let $S \subseteq R$ be the finite set of all the elements which occur as one of the Greek letters $\alpha, \ldots, \zeta$ in some subformula of type 1 through 4. 

Given any $S' \subseteq S$, let $\nu_{S'}$ be the product of all of the elements of $S'$ and let $\psi_{S'}$ be the formula $$\psi_{S'} := \bigwedge_{\xi \in S'} \xi(t) \neq 0 \wedge \bigwedge_{\xi \in S \setminus S'} \xi(t) = 0.$$

For each subset $S'$ of $S$, we will define a formula $\widetilde{\varphi}_{S'}(x_1, \ldots, x_n, y')$ such that 

\begin{equation}
\label{twiddle}
	\psi_{S'} \rightarrow \forall x_1 \ldots \forall x_n \forall y \left[\varphi(x_1, \ldots, x_n, y) \leftrightarrow \widetilde{\varphi}_{S'}(x_1, \ldots, x_n, \nu_{S'} \cdot y)  \right].
\end{equation}

 We construct $\widetilde{\varphi}_{S'}(x_1, \ldots, x_n, y)$ by modifying each atomic and negatomic subformula of $\varphi$ as follows:

\medskip

$\bullet$ For an atomic subformula $\alpha \cdot  y < a$ of type 1, first suppose that $\alpha \in S'$, in which case there is an $\alpha' \in R$ such that $\alpha \cdot \alpha' = \nu_{S'}$. In this case, we replace every instance of the subformula $\alpha \cdot y < a$ by $$\left[\alpha' > 0 \wedge ( \nu_{S'} \cdot y < \alpha' \cdot  a) \right] \vee [\alpha' < 0 \wedge (\nu_{S'} \cdot  y >  \alpha' \cdot a) ],$$ to which it is clearly equivalent if $\psi_{S'}$ holds. If $\alpha \notin S'$, then we replace $\alpha \cdot y < a$ by $0 < a$.

\medskip

$\bullet$ For an atomic subformula $b < \beta \cdot y$ of type 2, if $\beta \in S'$ and $\beta' \in R$ is chosen so that $\beta \cdot \beta' = \nu_{S'}$, then we replace every instance of it by the formula $$\left[\beta' > 0 \wedge (\beta' \cdot  b <    \nu_{S'} \cdot y) \right] \vee [\beta' < 0 \wedge ( \beta' \cdot b > \nu_{S'} \cdot  y)].$$ If $\beta \notin S'$, replace $b < \beta \cdot y$ by $b < 0$.

\medskip

$\bullet$ For an atomic subformula $D_\gamma( \delta \cdot y + c)$ of type 3, if $\delta \in S'$, then we can choose $\delta' \in R$ such that $\delta \cdot \delta' = \nu_{S'}$, and we replace each instance of $D_\gamma (\delta \cdot y + c)$ by the equivalent formula $$D_{\gamma \cdot \delta'}(\nu_{S'} \cdot y + \delta' \cdot c).$$ In case $\delta \notin S'$, replace $D_\gamma( \delta \cdot y + c)$ by $$D_\gamma( c).$$

\medskip

$\bullet$ A negatomic subformula $\neg D_\epsilon( \zeta \cdot y + d)$ of type 4 is dealt with just like an atomic subformula of type 3.

\medskip

So by the implication~(\ref{twiddle}), $$\psi_{S'} \rightarrow \forall x_1 \ldots \forall x_n \left[\exists y \varphi(x_1, \ldots, x_n, y) \leftrightarrow \exists y' (\widetilde{\varphi}_{S'}(x_1, \ldots, x_n, y')  \wedge D_{\nu_{S'}}(y'))\right] ,$$ and the normalized formula $\widetilde{\varphi}(x_1, \ldots, x_n, y')$ we seek is $$\bigvee_{S' \subseteq S} \psi_{S'} \wedge \widetilde{\varphi}_{S'}(x_1, \ldots, x_n, y')  \wedge D_{\nu_{S'}}(y').$$

\end{proof}

So from now on we will assume that the formula $\varphi(x_1, \ldots, x_n, y)$ is normalized in $y$.

\begin{definition}
\label{phi_inf}
Let $\varphi_{-\infty}(x_1, \ldots, x_n, y)$ be the formula obtained from $\varphi$ by replacing every atomic subformula of the form $y < a$ (of ``type 1'') with $0 = 0$ and every atomic subformula of the form  $b < y$ (of ``type 2'') with $0 \neq 0$.
\end{definition}

The idea behind $\varphi_{-\infty}$ is that for any fixed values of $t \in X$ and $k_1, \ldots, k_n \in \Z$, the truth values of $\varphi_t(k_1, \ldots, k_n, e)$ and $(\varphi_{-\infty})_t(k_1, \ldots, k_n, e)$ will coincide for all sufficiently small $e \in \Z$. The following is straightforward:

\begin{lem}
\label{phi_inf_values}
There is a function $M: \Z^n \times X \rightarrow \Z$ such that for every $e \in \Z$, if $e < M(k_1, \ldots, k_n, t)$, then $$\models \varphi_t(k_1, \ldots, k_n, e) \,  \Leftrightarrow \, \models (\varphi_{-\infty})_t(k_1, \ldots, k_n, e). $$
\end{lem}

Next we define a parametrized family of sets of terms $\{ \mathcal{B}_t : t \in X\}$ which will include all possible values of the lower bounds $b$ in the atomic subformulas $b < y$ of type 2 in $\varphi$. The definition is recursive, and note that we only have to deal with positive Boolean combinations since $\varphi$ is normalized in $y$.

\begin{definition}
\label{B_set}
 If $\psi$ is a basic formula $b < y$ of type (2) as in Definition~\ref{prepared}, then $\mathcal{B}_t(\psi) = \{b\}$; if $\psi$ is a basic formula of type (1), (3), (4), or (5), then $\mathcal{B}_t(\psi) = \emptyset$. Given normalized formulas $\psi_1$ and $\psi_2$, $$\mathcal{B}_t(\psi_1 \vee \psi_2) = \mathcal{B}_t(\psi_1) \cup \mathcal{B}_t(\psi_2)$$ and $$\mathcal{B}_t(\psi_1 \wedge \psi_2) = \mathcal{B}_t(\psi_1) \cup \mathcal{B}_t(\psi_2).$$ Given a normalized formula $\psi(x_1, \ldots, x_n, z)$ and letting $\tilde{c}$ be an $\mathcal{L}_{Pres}$-term representing $c \in \Z$, $$\mathcal{B}_t\left(\bigvee_{z=0}^{\alpha(t)} \psi(x_1, \ldots, x_n, z)\right) = \bigcup_{c=0}^{\alpha(t)} \{s(x_1, \ldots, x_n, \widetilde{c}): s \in \mathcal{B}_t(\psi) \textup{ and } 0 \leq c \leq \alpha(t) \}$$ and $$\mathcal{B}_t\left(\bigwedge_{z=0}^{\alpha(t)} \psi(x_1, \ldots, x_n, z)\right) = \bigcup_{c=0}^{\alpha(t)} \{s(x_1, \ldots, x_n, \widetilde{c}): s \in \mathcal{B}_t(\psi) \textup{ and } 0 \leq c \leq \alpha(t) \}.$$ Let $\mathcal{B}_t$ stand for $\mathcal{B}_t(\varphi)$. We will sometimes omit the superscript $t$ below.
\end{definition}

The next Lemma is immediate from the last definition:

\begin{lem}
\label{B_form}
The set $\mathcal{B}_t$ can be written in the form $$\mathcal{B}_t = \{s_i(x_1, \ldots, x_n; \tilde{c}_1, \ldots, \tilde{c}_m) : 0 \leq i \leq k, 0 \leq c_j \leq \alpha_j(t)\},$$ where $s_0, \ldots, s_k$ are fixed $\mathcal{L}^+_R$-terms and $\alpha_0, \ldots, \alpha_m$ are fixed elements of $R$.
\end{lem}

By Lemma~\ref{B_form}, we are justified in writing formulas such as $$\bigvee_{s \in \mathcal{B}} \theta(s),$$ which stands for the $\mathcal{L}^+_R$-formula $$\bigvee_{i=0}^k \bigvee_{z_0 = 0}^{\alpha_0(t)} \ldots \bigvee_{z_m=0}^{\alpha_m(t)} \theta(s_i(x_1, \ldots, x_n, y, z_0, \ldots, z_m)),$$ and if $\theta$ has $R$-bounded quantifiers then so will the displayed formula above.

Let $S \subseteq R$ be the set of all values of $\alpha$ or $\beta$ which occur in subformulas of $\varphi$ of the form $D_\alpha(y +c)$ (``type 3'') or $\neg D_\beta(y+d)$ (``type 4''). If $S_-$ and $S_+$ are disjoint subsets of $S$, we define the formula $$\psi_{S_-, S+} := \bigwedge_{\xi \in S_-} \xi(t) < 0 \wedge \bigwedge_{\xi \in S_+} \xi(t) > 0 \wedge \bigwedge_{\xi \in S \setminus (S_- \cup S_+)} \xi(t) = 0,$$ and let $$\delta_{S_-, S_+}(t) = \pm \prod_{\xi \in S_- \cup S_+} \xi(t) \in R$$ with the sign chosen so that $\psi_{S_-, S_+} \models \delta_{S_-, S_+}(t) > 0$. By convention, let $\delta_{\emptyset, \emptyset} = 1$.

\begin{lem}
\label{boundary}
If $t \in X$, $k_1, \ldots, k_n, e \in \Z$, and $$\models (\psi_{S_-, S_+})_t \wedge \varphi_t(k_1, \ldots, k_n, e) \wedge \neg \varphi_t(k_1, \ldots, k_n, e - \delta_{S_-, S_+}(t)),$$ then $$e \in \left\{s(k_1, \ldots, k_n) + \ell : 1 \leq \ell \leq \delta_{S_-, S_+}(t) \textup{ and } s \in \mathcal{B}_t \right\}.$$
\end{lem}

\begin{proof}
We fix $(k_1, \ldots, k_n) \in \Z^n$ and $t \in X$ such that $\models (\psi_{S_-, S_+})_{t}$ and we consider how replacing $e$ with $e - \delta_{S_-, S_+}(t)$ for the value of $y$ affects the truth value of each of the four types of basic formulas from Definition~\ref{prepared}. For a formula of the type $D_\alpha (y + c)$ (type 3), either $\alpha \notin S_- \cup S_+$ and $\alpha(t) = 0$, in which case this formula is always false, or else $\alpha \in S_- \cup S_+$ and $$\models D_{\alpha(t)}(e + c(k_1, \ldots, k_n)) \Leftrightarrow \models D_{\alpha(t)}(e - \delta_{S_-, S_+}(t) + c(k_1, \ldots, k_n))$$ since $\alpha(t)$ divides $\delta_{S_-, S_+}(t)$. By the same argument, the truth value of subformulas of the form $\neg D_\beta(y+d)$ is unchanged by replacing $e$ by $e - \delta_{S_-, S_+}(t)$. Since $\delta_{S_-, S_+}(0) > 0$, any basic subformula of the form $y < a$ (type 1) can only change from being false to being true as $y$ changes from $e$ to $e - \delta_{S_-, S_+}(t)$, but as $\varphi$ is a positive Boolean combination of basic subformulas, the truth value of $\varphi_t(k_1, \ldots, k_n, e)$ itself can only change from false to true if we replace $e$ by $e - \delta_{S_-, S_+}(t)$. So for the hypothesis of the Lemma to hold, some basic subformula of the form $b < y$ (type 2) must change from true to false as we change the value of $y$ from $e$ to $e - \delta_{S_-, S_+}(t)$, or in other words, $e \in \{b+1, b+2, \ldots, b + \delta_{S_-, S_+}(t)\}$.  Since the terms in  $\mathcal{B}_t$ represent all possible values of such a term $b$, the Lemma follows.
\end{proof}

Finally, we can describe the $\mathcal{L}^+_R$-formula with $R$-bounded quantifiers that is logically equivalent to $\exists y \left[\varphi(\overline{x}, y) \right]$.

\begin{lem}
\label{exists_elimination}
The formula $\exists y \left[\varphi(x_1, \ldots, x_n, y) \right]$ is logically equivalent to

\begin{equation*} \label{eq:elimination} 
\tag{$\star$}
\bigvee_{S_-, S_+ \subseteq S} \left[ \psi_{S_-, S_+} \wedge \left(\bigvee_{z=1}^{\delta_{S_-, S_+}(t)} \varphi_{-\infty}(x_1, \ldots, x_n, z)  \vee \bigvee_{z=1}^{\delta_{S_-, S_+}(t)} \bigvee_{s \in \mathcal{B}} \varphi(x_1, \ldots, x_n, s + z) \right) \right].
\end{equation*} 

\end{lem}

Of course it may be the case that $\varphi$ contains no basic subformulas of type 2, in which case $\mathcal{B} = \emptyset$ and the second big disjunction within the brackets will be empty, and this ``empty disjunction'' should be interpreted as $0 \neq 0$.

\begin{proof}
Throughout the proof, we will assume that the value of the parameter $t$ is fixed and that the variables $(x_1, \ldots, x_n)$ also have fixed values $(k_1, \ldots, k_n) \in \Z^n$. In the proof below, we will sometimes write simply ``$\theta$'' instead of $\theta_t$ for ease of reading.

Let $(S_-, S_+)$ be the unique pair of subsets of $S$ such that $\models (\psi_{S_-, S_+})_t$.

First, suppose that the formula \eqref{eq:elimination} of the Lemma is true, and we must show that $\exists y \left[\varphi(k_1, \ldots, k_n, y) \right]$ also holds. One possibility is that $\varphi_{-\infty}(k_1, \ldots, k_n, \ell)$ is true for some $\ell \in \{1, \ldots, \delta_{S_-, S_+}(t)\}$. In this case, since the basic subformulas that occur in $\varphi_{-\infty}$ all involve predicates $D_{\alpha(t)}(x)$ whose truth values are invariant under adding multiples of $\alpha(t)$ to $x$ (since either $\alpha(t)$ divides $\delta_{S_-, S_+}(t)$ or else $\alpha(t) = 0$ and $D_{\alpha(t)}$ is always false), we have that for any $r \in \Z$, $$\models \varphi_{-\infty}(k_1, \ldots, k_n, \ell + r \delta_{S_-, S_+}(t)).$$ By Lemma~\ref{phi_inf_values}, if $r \ll 0$, then $\models \varphi(k_1, \ldots, k_n, \ell + r \delta_{S_-, S_+}(t))$ and hence $\models \exists y \left[\varphi(k_1, \ldots, k_n, y) \right]$. The only other way that \eqref{eq:elimination} could be true is if the second large disjunction within the brackets is true, but this immediately implies that $ \exists y \,\varphi(k_1, \ldots, k_n, y)$ holds.

Conversely, assume that $\models \varphi(k_1, \ldots, k_n, e)$ for some $e \in \Z$, and we must show that the formula \eqref{eq:elimination} is true.

\textbf{Case 1: $\models \varphi(k_1, \ldots, k_n, e - r \delta_{S_-, S_+}(t))$ holds for every $r \in \N$.} Then by Lemma~\ref{phi_inf_values}, for some sufficiently large $r \in \N$,  $\models \varphi_{-\infty}(k_1, \ldots, k_n, e - r \delta_{S_-, S_+}(t))$ holds. But the truth value of $\varphi_{-\infty}$ is unchanged by adding multiples of $\delta_{S_-, S_+}(t)$ to the last coordinate, since it only involves basic subformulas with $D_\alpha$ or $\neg D_\beta$, and therefore $\models \varphi_{-\infty}(k_1, \ldots, k_n, e - r' \delta_{S_-, S_+}(t))$ holds for \emph{every} $r' \in \Z$. We can pick $r'$ such that $e - r' \delta_{S_-, S_+}(t) \in \{1, \ldots, \delta_{S_-, S_+}(t)\}$, giving a witness for $z$ in $$\bigvee_{z=1}^{\delta_{S_-, S_+}(t)} \varphi_{-\infty}(k_1, \ldots, k_n, z),$$ so the formula \eqref{eq:elimination} holds.

\textbf{Case 2: For some $r \in \N$, $$\models \varphi(k_1, \ldots, k_n, e - r \delta_{S_-, S_+}(t)) \wedge \neg \varphi(k_1, \ldots, k_n, e - (r+1) \delta_{S_-, S_+}(t)).$$} Then by Lemma~\ref{boundary}, $$e - r \delta_{S_-, S_+}(t) \in \{s(k_1, \ldots, k_n) + \ell : 1 \leq \ell \leq \delta_{S_-, S_+}(t) \textup{ and } s \in \mathcal{B}_t \}.$$ So in this case the second big disjunction within the brackets of \eqref{eq:elimination} is true, and we are done.

\end{proof}

\textit{Proof of Theorem~\ref{qe_main}:} By a routine inductive argument on the number of quantifiers, to prove Theorem~\ref{qe_main}, it suffices to show the following:

\begin{claim}
\label{qe_reduction}
For any $\mathcal{L}^+_R$-formula with $R$-bounded quantifiers $\varphi(x_1, \ldots, x_n, y)$, the formula $\exists y \left[\varphi(x_1, \ldots, x_n, y) \right]$ is logically equivalent to some $\mathcal{L}^+_R$-formula with $R$-bounded quantifiers.
\end{claim}

But this Claim follows immediately from Lemma~\ref{exists_elimination}. $\square$

\subsection{Proof of Theorem~\ref{qe_criterion}}

The proof of Theorem~\ref{qe_criterion} follows the same procedure as that of Theorem~\ref{qe_main} above, so we will briefly indicate the necessary changes to show Theorem~\ref{qe_criterion}.

Observe that to prove Theorem~\ref{qe_criterion}, it suffices to show the following:

\begin{claim}
\label{singe_e_elimination}
If $\varphi(\overline{x}, \overline{y})$ is any quantifier-free $\mathcal{L}^+_R$-formula, $\overline{y} = (y_1, \ldots, y_m)$ is a distinguished subtuple of the free variables, and $\varphi$ satisfies condition (1) of Theorem~\ref{qe_criterion} plus

\bigskip

\hangindent=0.3in $(2)^\prime$ $\varphi$ does not contain any term $f_\alpha(s)$ where $s$ is a term containing one of the variables $y_i$ and $\alpha \in R$ is not a constant,\\

\hangindent=0in then $\exists y_1 \varphi(\overline{x}, \overline{y})$ is equivalent to a quantifier-free $\mathcal{L}^+_R$-formula which also satisfies conditions (1) and $(2)^\prime$.

\end{claim}

\begin{proof}

The first step is to apply the procedure of Lemma~\ref{prepared} to show that $\exists y_1 \varphi(\overline{x}, \overline{y})$ is logically equivalent to a formula $\exists y' \widetilde{\varphi}(\overline{x}, y', y_2, \ldots, y_m)$ which is normalized in $y'$ and such that $\widetilde{\varphi}$ still satisfies (1) and $(2)^\prime$. The crucial point here is that the hypotheses  (1) and $(2)^\prime$ on $\varphi$ imply that all the elements $\alpha, \ldots, \theta$ of $R$ in subformulas of type (1) through (6) must all be constants, so the new divisibility conditions $D_{\nu'}$ which occur in $\widetilde{\varphi}$ also involve only constant elements $\nu'$ of $R$.

Then, Lemma~\ref{exists_elimination} can be applied to the normalized formula $\widetilde{\varphi}$, with the additional condition that each quantity $\delta_{S_-, S_+}$ is constant since it is a product of constants $\alpha$ which occur in atomic subformulas $D_\alpha( \cdot)$. Furthermore, the disjunction over $\mathcal{B}$ in the formula \eqref{eq:elimination} is also of bounded size since we are assuming that $\varphi$ is quantifier-free (the size of $\mathcal{B}$ grows with $t$ only if there are bounded quantifiers in $\varphi$). So all the disjunctions in the formula \eqref{eq:elimination} are of constant size, and this gives a quantifier-free formula equivalent to $\exists y_1 \varphi(\overline{x}, \overline{y})$.

\end{proof}

\bibliography{parPres}

\providecommand{\bysame}{\leavevmode\hbox to3em{\hrulefill}\thinspace}
\providecommand{\MR}{\relax\ifhmode\unskip\space\fi MR }
\providecommand{\MRhref}[2]{%
  \href{http://www.ams.org/mathscinet-getitem?mr=#1}{#2}
}
\providecommand{\href}[2]{#2}
\begin{thebibliography}{1}

\bibitem{BGW}
John~Goodrick Bogart, Tristram and Kevin Woods, \emph{Parametric {P}resburger
  arithmetic: logic, combinatorics, and quasi-polynomial behavior}, Discrete
  Analysis (2017).

\bibitem{cooper}
D.~C. Cooper, \emph{Theorem proving in arithmetic without multiplication},
  Machine Intelligence \textbf{7} (1972), 91--100.

\bibitem{ehrhart}
Eug\`ene Ehrhart, \emph{Sur les poly\`edres rationnels homoth\'etiques \`a n
  dimensions}, Comptes Rendues de l'Acad\'emie des Sciences, Paris \textbf{254}
  (1962), 616--618.

\bibitem{glivicky}
Petr Glivick\'y, \emph{Study of arithmetical structures and theories with
  regard to representatives and descriptive analysis}, Ph.D. Thesis, Charles
  University in Prague, 2013.

\bibitem{lasaruksturmweakQE}
Aless Lasaruk and Thomas Sturm, \emph{Weak quantifier elimination for the full
  linear theory of the integers}, Applicable Algebra in Engineering,
  Communication and Computing \textbf{18} (2007), no.~6, 545--574.

\bibitem{roune_woods}
Bjarke~Hammersholt Roune and Kevin Woods, \emph{The parametric frobenius
  problem}, Electronic Journal of Combinatorics \textbf{22} (2015).

\bibitem{weis}
Volker Weispfenning, \emph{Complexity and uniformity of elimination in
  {P}resburger arithmetic}, Proceedings of the 1997 International Symbosium on
  Symbolic and Algebraic Computation, ACM Press, 1997, pp.~48--53.

\bibitem{woods1}
Kevin Woods, \emph{The unreasonable ubiquitousness of quasi-polynomials},
  Electronic Journal of Combinatorics \textbf{21} (2014).

\end{thebibliography}

\end{document}